\documentclass[11pt]{amsart}

\usepackage{amsmath,amssymb,amsthm}
\usepackage[foot]{amsaddr}
\usepackage{bm}
\usepackage{mathrsfs}
\usepackage[all]{xy}
\usepackage{booktabs}
\usepackage{fullpage}
\usepackage{color}
\usepackage[breaklinks]{hyperref}
\usepackage{mathtools}

\newcommand{\dba}{\bar{\partial}}

\usepackage{ulem}

\usepackage{url}

\usepackage{breakurl}
\usepackage[abbrev,msc-links]{amsrefs}
\hypersetup{
 hidelinks,
 bookmarksopen=true,
}

\theoremstyle{plain}
\newtheorem{theorem}{Theorem}[section]

\begin{document}

\title{Uniform $L^2$-estimates for flat nontrivial line bundles on compact complex manifolds}

\author{Yoshinori Hashimoto$^1$}

\email{yhashimoto@omu.ac.jp}
\address{$^{1,2}$Department of Mathematics, Osaka Metropolitan University, 3-3-138, Sugimoto, Sumiyoshi-ku, Osaka, 558-8585, Japan.}

\author{Takayuki Koike$^2$}
\email{tkoike@omu.ac.jp}

\author{Shin-ichi Matsumura$^3$}
\email{mshinichi-math@tohoku.ac.jp}
\email{mshinichi0@gmail.com}
\address{$^3$Mathematical Institute, Tohoku University, 6-3, Aramaki Aza-Aoba, Aoba-ku, Sendai 980-8578, Japan.}
\date{\today}

\begin{abstract}
In this paper, we extend the uniform $L^2$-estimate of $\dba$-equations for flat nontrivial line bundles, 
proved for compact K\"ahler manifolds in \cite{KH22}, 
to compact complex manifolds. 
In the proof, by tracing the Dolbeault isomorphism in detail, 
we derive the desired $L^2$-estimate directly from Ueda's lemma.
\end{abstract}

\maketitle

\tableofcontents

\section{Introduction}
This paper studies the uniform estimate of the $L^2$-norm $\|u \|$ of the solution of
the $\dba$-equation $\dba u =v$ for flat nontrivial line bundles. 
We begin by recalling the uniform $L^2$-estimate on compact K\"ahler manifolds
proved by the first two authors in \cite{KH22}. 
Let $X$ be a compact complex manifold, 
and let $\mathcal{P} (X)$ be the set of all unitarily flat (holomorphic) line bundles on $X$, 
which can be regarded as an abelian group 
under the tensor product and is compact with respect to a group-invariant distance $\mathsf{d}$ 
as defined in e.g.~\cite{KH22}*{Lemma 2.6} and \cite{Ueda}*{\S 4.5} (see also \cite{KoikeUehara}*{\S A.3}). 
It is well-known that $\mathcal{P} (X)$ is a direct sum of copies of the Picard variety $\mathrm{Pic}^0 (X)$ 
(see e.g.~\cite{KH22}*{Lemma 2.2}).
We write $\mathbb{I}_X$ for the trivial line bundle on $X$, which is the identity in $\mathcal{P}(X)$. 
Then, we have:

\begin{theorem} [{\cite{KH22}*{Theorem 1.1}}] \label{thm:main_1}
Let $(X, g)$ be a compact K\"ahler manifold. 
Then, there exists a constant $K>0$ such that, 
for any element $F\in \mathcal{P}(X)\setminus\{\mathbb{I}_X\}$ and 
any smooth $\bar{\partial}$-closed $(0, 1)$-form $v$ with values in $F$ 
whose Dolbeault cohomology class $[v]\in H^{0, 1}(X, F)$ is trivial,
there exists a unique smooth section $u$ of $F$ such that $\bar{\partial} u = v$ 
and 
\[
\sqrt{\int_X|u|_{h}^2\,dV_g} \leq \frac{K}{\mathsf{d}(\mathbb{I}_X, F)} \sqrt{\int_X|v|_{h, g}^2\,dV_g}
\]
hold for a flat Hermitian metric $h$ on $F$. 
\end{theorem}

This result can be regarded as an $L^2$-version of Ueda's lemma. 
Ueda's lemma (discussed below) is a foundational result in complex dynamics and complex analysis.

\begin{theorem}[Ueda's lemma, {\cite{Ueda}*{Lemma 4}}]\label{cor:main}
Let $X$ be a compact complex manifold and $\mathcal{U}:=\{U_j\}$ be a finite open covering of $X$ 
such that each $U_j $ is a Stein open set and $U_j$  trivializes any $F\in\mathcal{P}(X)$. 
Then, there exists a constant $K_{\mathcal{U}} >0$ such that, for any $F\in\mathcal{P}(X) \setminus\{\mathbb{I}_X\}$ and any \v{C}ech $0$-cochain $\mathfrak{f}:=\{(U_j, f_{j})\}\in \check{C}^0(\mathcal{U}, \mathcal{O}_X(F))$, the inequality
\[
\max_j\sup_{U_j}|f_j|_h\leq \frac{K_{\mathcal{U}}}{\mathsf{d}(\mathbb{I}_X, F)} \max_{j, k}\sup_{U_{jk}}|f_{jk}|_h
\]
holds for a flat Hermitian metric $h$ on $F$, 
where $\{(U_{jk}, f_{jk})\}:=\delta \mathfrak{f}\in \check{C}^1(\mathcal{U}, \mathcal{O}_X(F))$ is the \v{C}ech coboundary of $\mathfrak{f}$. 
\end{theorem}

This can be partially generalized to higher degree cohomology classes,
 leading to a cohomology vanishing result \cite{KH22}*{Theorem 1.2}, 
 which seems to be related to the generic vanishing theorem (see e.g.~\cite{Lazarsfeld1}*{\S 4.4}).

As written in \cite{KH22}*{\S 2.3}, Theorem \ref{thm:main_1} implies Theorem \ref{cor:main} 
when $X$ is a compact K\"ahler manifold. 
It is natural to speculate that Theorem \ref{thm:main_1} may hold for a compact complex manifold 
that is not necessarily K\"ahler. We affirmatively answer this question, as follows:

\begin{theorem}\label{thm:main_2}
Theorem \ref{thm:main_1} holds for a compact complex manifold $(X, g)$ with a Hermitian metric $g$.
\end{theorem}

An important distinction from Theorem \ref{thm:main_1} is that the proof of the above theorem crucially depends on Ueda's lemma (Theorem \ref{cor:main}); 
Indeed, Theorem \ref{thm:main_2} is directly derived from Ueda's lemma. 
One of the motivations behind Theorem \ref{thm:main_1} was to find a new proof of Ueda's lemma in such a way that it is more geometric compared to the original argument in \cite{Ueda}, which was rather technical. 
Here, we establish a direction that is somewhat contrary to the one discussed in \cite{KH22}, 
with the merit of being able to extend \cite{KH22}*{Theorem 1.1} to general compact \textit{complex} manifolds.

\subsection*{Acknowledgements}
The authors would like to thank the organizers 
of the 7th workshop \lq \lq Complex Geometry and Lie Groups," 
which served as an impetus for this collaboration. 

The first named author thanks Osamu Fujino, Hisashi Kasuya, and Shinnosuke Okawa for helpful discussions, and is supported by Grant-in-Aid for Scientific Research (C) $\sharp$23K03120 and Grant-in-Aid for Scientific Research (B) $\sharp$24K00524.
The second named author is supported by Grant-in-Aid for Scientific Research (C) $\sharp$23K03119.
The third named author is supported 
by  Grant-in-Aid for Scientific Research (B) $\sharp$21H00976 from JSPS.

\section{Proof of Theorem \ref{thm:main_2}}
This section is devoted to the proof of Theorem \ref{thm:main_2}. 
In what follows, we adhere to the notation used in Theorem \ref{thm:main_2} and \cite{KH22}.


\begin{proof}[Proof of Theorem \ref{thm:main_2}]

The strategy of the proof is to chase the Dolbeault isomorphism between $\dba$-cohomology and \v{C}ech cohomology, 
using the H\"ormander's $L^2$-method, which allows us to apply Ueda's lemma for the $\delta$-equation. 

We work with a fixed Hermitian metric $g$ on $X$, with all norms and volume forms defined with respect to $g$.
We now recall a foundational result by H\"ormander 
(see e.g.~\cite{Demailly_book}*{Chapter VIII, Theorem 6.9}, \cite{Dem82}*{5, 4.1 Th\'eor\`eme}, 
or \cite{Hormbook}*{Theorem 4.4.2}), which guarantees the following: 
Let $U^{*}_x$ be a Stein open neighborhood of a given point $x \in X$,  
and let $h=e^{-\varphi}$ be a singular Hermitian  metric 
on the trivial line bundle with $\varphi$ psh. 
Then, for any smooth $(0, 1)$-form $v_x$ on $U^*_x$ such that 
$\bar{\partial} v_x=0$ and $\Vert v_x \Vert_{L^2 (U^{*}_x), h}<\infty$, 
there exists a smooth function $u_x$ on $U^{*}_x$ such that 
\begin{equation*}
	\bar{\partial} u_x = v_x \text { on } U_x^*
\end{equation*}
 with the $L^2$ estimate
\begin{equation*}
	\Vert u_x \Vert_{L^2(U^*_x), h} \le C_{x,g} \Vert v_x \Vert_{L^2 (U^*_x), h},
\end{equation*}
where the constant $C_{x,g} >0$ depends only on $U^*_x$ and $g$. 
Note that the solution of $u_x$ can be chosen to be smooth if $h=e^{-\varphi}$ and $v_x$ are smooth.

We may assume, without loss of generality and by shrinking each $U^*_x$, that $U^*_x$ trivializes 
any flat line bundle on $X$ (see \cite{KH22}*{Lemma 2.7}). 
For each $x \in X$, we pick a relatively compact open subset $U^{**}_x \Subset U^*_x$. 
We then cover $X$ with the open sets $\{ U^{**}_x \}_{x \in X}$, and by the compactness of $X$, 
we may select a finite subcover $\{ U^{**}_{x_j} \}_{j \in I}$. 
By setting $U'_j := U^{**}_{x_j}$ and $U_j := U^{*}_{x_j}$ for each $j \in I$, 
we obtain two finite covers $\{ U'_j \}_{j \in I}$ and $\{ U_j \}_{j \in I}$ of $X$. 
By construction, we can observe that $U_j$ are Stein, 
$U'_j \Subset U_j$, and $U_j$ (and thus $U'_j$ as well) trivializes any flat line bundle on $X$.

Let $F$ be a flat line bundle endowed with a flat Hermitian  metric $h$, 
which is unique up to a multiplicative constant (see e.g.~\cite{KH22}*{Lemma 2.4}). 
Let $v$ be a smooth $\bar{\partial}$-closed $(0, 1)$-form with values in $F$, whose Dolbeault cohomology class $[v] \in H^{0, 1}(X, F)$ is trivial. 
In what follows, all norms are with respect to the fixed flat Hermitian  metric $h$ on $F$ and the fixed Hermitian  metric $g$ on $X$. 
By applying H\"ormander's estimates and trivializing $F$ equipped with the metric $h$, 
we find that for each $j \in I$ there exists a smooth function $u_j$ on $U_j$ such that
\begin{equation}\label{eqlchmdesfl}
	\bar{\partial} u_j = v |_{U_j} \text{ and } 
	\Vert u_j \Vert_{L^2(U_j),h} \le C^{(1)}_{j,g} \Vert v \Vert_{L^2 (U_j),h}
\end{equation}
for a constant $C^{(1)}_{j,g} > 0$, which depends only on $U_j$ and $g$. 
The key point of this estimate is that it holds uniformly for all flat line bundles $(F,h)$ 
and for all $v$ representing a trivial Dolbeault class in $H^{0, 1}(X, F)$.

We then find that $\{ u_i - u_j \}_{i,j \in I}$ defines a holomorphic section of $F$ on $U_{ij} := U_i \cap U_j$. The isomorphism between \v{C}ech cohomology and Dolbeault cohomology implies that $\{ u_i - u_j \}_{i,j\in I}$ represents the \v{C}ech cohomology class corresponding to $[v]$. Since $[v]$ is trivial, we find that $\{ u_i - u_j \}_{i,j \in I}$ is a \v{C}ech $1$-coboundary (after replacing $U^{*}_x$ with a smaller polydisk if necessary), meaning that there exist local holomorphic sections $\{ f_j \}_{j \in I}$ of $F$ such that 
$$f_i - f_j = u_i - u_j$$ on $U_{ij}$ for all $i, j \in I$. We take a partition of unity $\{ \rho_j \}_{j \in I}$ subordinate to $\{ U_j \}_{j \in I}$. The equation $f_i - f_j = u_i - u_j$ on $U_{ij}$ for all $i, j \in I$ then implies
\begin{equation*}
	\sum_{j \in I} \rho_j (u_i - u_j) = \sum_{j \in I} \rho_j (f_i - f_j)
\end{equation*}
on $U_i$. Taking the $\bar{\partial}$-operator of both sides, we have
\begin{equation*}
	\bar{\partial} u_i - \bar{\partial} \left( \sum_{j \in I} \rho_j u_j \right) = -\bar{\partial} \left( \sum_{j \in I} \rho_j f_j \right)
\end{equation*}
on $U_i$, as $\sum_{j \in I} \rho_j \equiv 1$. Since $\bar{\partial} u_i = v |_{U_i}$, we get
\begin{equation*}
	v |_{U_i} = \bar{\partial} \left( \sum_{j \in I} \rho_j (u_j - f_j) |_{U_i} \right).
\end{equation*}
Noting that $\sum_{j \in I} \rho_j (u_j - f_j)$ is defined globally on $X$, we have
\begin{equation*}
	v = \bar{\partial} \left( \sum_{j \in I} \rho_j (u_j - f_j) \right),
\end{equation*}
and hence $u := \sum_{j \in I} \rho_j (u_j - f_j)$ gives the solution to the equation $\bar{\partial} u = v$, which is necessarily unique since $F$ is flat and nontrivial (a well-known result, see e.g.~\cite{KH22}*{Lemma 2.3} for the proof).

\smallskip

The argument above (except for the uniqueness of solutions)
is valid for any $(0,q)$-forms $v$ with values in $F\otimes \mathcal{I}(h)$, 
where $h$ is a singular Hermitian  metric on $F$ with positive curvature current 
(see \cite{Mat}*{Subsection 5.3} for details).
In the following discussion, we essentially use the flatness property.

We apply Ueda's lemma (Theorem \ref{cor:main}) to the \v{C}ech $0$-cochain $\mathfrak{f}:=\{(U'_j, f_{j} |_{U'_j})\}$ corresponding to the smaller cover $\mathcal{U}':= \{ U'_j \}_{j \in I}$, to obtain
\begin{equation} \label{equdfjfjk}
	\max_{j \in I} \sup_{U'_j}|f_j|_h \leq \frac{K_{\mathcal{U}'}}{\mathsf{d}(\mathbb{I}_X, F)} \max_{j, k \in I}\sup_{U'_{jk}}|f_{jk}|_h
\end{equation}
uniformly for all $F \in \mathcal{P} (X) \setminus \{ \mathbb{I}_X \}$, for some constant $K_{\mathcal{U}'} > 0$. Note that we have
\begin{equation} \label{eqfjl2lif}
	\Vert f_j \Vert_{L^2(U'_j)} \leq \mathrm{Vol}(U'_j)^{1/2} \max_{j \in I} \sup_{U'_j}|f_j|_h
\end{equation}
for each $j \in I$, and
\begin{equation*}
	\max_{j, k \in I}\sup_{U'_{jk}}|f_{jk}|_h = \max_{j, k \in I}\sup_{U'_{jk}}|f_j - f_k|_h = \max_{j, k \in I}\sup_{U'_{jk}}|u_j - u_k|_h
\end{equation*}
since $u_j - u_k = f_j - f_k$. 
Consider a weight $\varphi_j$ of $h$ on $U_j$ such that $h = e^{-\varphi_j}$. 
Since $\varphi_j$ is pluriharmonic, there exists a holomorphic function $w_j$ on $U_j$ 
such that $\varphi_j$ is the real part of $w_j$, which implies that $h = e^{-\varphi_j} = |e^{-w_j}|$. 
Given that $(u_j - u_k)e^{-w_j}$ is holomorphic on $U_{jk}$ (which contains $U'_{jk}$),
we have
\begin{equation*}
	\sup_{U'_{jk}}|u_j - u_k|_h \leq C^{(2)}_{j,g} \Vert u_j - u_k \Vert_{L^2(U_{jk}),h} \leq 2 C^{(2)}_{j,g} \Vert u_j \Vert_{L^2(U_{jk}),h}
\end{equation*}
for some constant $C^{(2)}_{j,g} > 0$, depending only on $U'_j$, $U_j$, and $g$, 
by the mean value inequality and the Cauchy--Schwarz inequality. 
Hence
\begin{align}
	\max_{j, k \in I}\sup_{U'_{jk}}|f_{jk}|_h = \max_{j, k \in I}\sup_{U'_{jk}}|u_j - u_k|_h &\le 2 \max_{j, k \in I} C^{(2)}_{j,g}\Vert u_j \Vert_{L^2(U_{jk}),h} \notag \\
	&\le 2 \max_{j \in I} C^{(2)}_{j,g} \Vert u_j \Vert_{L^2(U_{j}),h} \notag \\
	&\le 2 \max_{j \in I} C^{(2)}_{j,g} C^{(1)}_{j,g} \Vert v \Vert_{L^2(U_{j}),h} \le C^{(3)}_{g} \Vert v \Vert_{L^2(X),h} \label{eqfjklifvl2}
\end{align}
by the H\"ormander estimate (\ref{eqlchmdesfl}) on $U_j$, where we set $C^{(3)}_{g} := 2 \max_j C^{(2)}_{j,g} C^{(1)}_{j,g} >0$. With (\ref{equdfjfjk}) and (\ref{eqfjl2lif}), the estimate (\ref{eqfjklifvl2}) gives
\begin{equation} \label{eql2fj}
	\Vert f_j \Vert_{L^2(U_j),h} \le \sum_{k \in I} \Vert f_k \Vert_{L^2(U'_k),h} \le  \frac{ K_{\mathcal{U}'} C^{(3)}_{g}}{\mathsf{d}(\mathbb{I}_X, F)} \left( \sum_{k \in I} \mathrm{Vol}(U'_k)^{1/2} \right) \Vert v \Vert_{L^2(X),h}
\end{equation}
for any $j \in I$.

We now recall $u= \sum_{j \in I} \rho_j (u_j - f_j)$ and evaluate
\begin{equation*}
	\Vert u \Vert_{L^2(X),h} \le C^{(4)}_{g} \sum_{j \in I} \left(  \Vert u_j \Vert_{L^2(U_j),h} + \Vert f_j \Vert_{L^2(U_j),h} \right) \le C^{(4)}_{g} \sum_{j \in I} \left(  C^{(1)}_{j,g} \Vert v \Vert_{L^2(X),h} + \Vert f_j \Vert_{L^2(U_j),h} \right) 
\end{equation*}
again by the H\"ormander estimate (\ref{eqlchmdesfl}) on $U_j$, where $C^{(4)}_g >0$ is a constant which depends only on $g$. Combined with (\ref{eql2fj}), we thus get
\begin{equation*}
	\Vert u \Vert_{L^2(X),h} \le \frac{C_g^{(4)}}{\mathsf{d}(\mathbb{I}_X, F)} \left( \sum_{j \in I} \left( C^{(1)}_{j,g} \mathsf{d}(\mathbb{I}_X, F) + K_{\mathcal{U}'} C^{(3)}_{g} \sum_{k \in I} \mathrm{Vol}(U'_k)^{1/2} \right) \right) \Vert v \Vert_{L^2(X),h}.
\end{equation*}
Noting that $\sup_{F \in \mathcal{P}(X)} \mathsf{d}(\mathbb{I}_X, F)$ is finite and $I$ is a finite set, we get the required estimate.
\end{proof}

\bibliography{ueda.bib}

\begin{bibdiv}
\begin{biblist}

\bib{Dem82}{article}{
      author={Demailly, Jean-Pierre},
       title={Estimations {$L\sp{2}$}\ pour l'op\'erateur {$\bar \partial $}\
  d'un fibr\'e{} vectoriel holomorphe semi-positif au-dessus d'une
  vari\'et\'e{} k\"ahl\'erienne compl\`ete},
        date={1982},
        ISSN={0012-9593},
     journal={Ann. Sci. \'Ecole Norm. Sup. (4)},
      volume={15},
      number={3},
       pages={457\ndash 511},
         url={http://www.numdam.org/item?id=ASENS_1982_4_15_3_457_0},
      review={\MR{690650}},
}

\bib{Demailly_book}{misc}{
      author={Demailly, Jean-Pierre},
       title={Complex analytic and differential geometry},
        date={2012},
  note={\url{https://www-fourier.ujf-grenoble.fr/~demailly/manuscripts/agbook.pdf}},
}

\bib{KH22}{article}{
      author={Hashimoto, Yoshinori},
      author={Koike, Takayuki},
       title={{U}eda's lemma via uniform {H}{\"o}rmander estimates for flat
  line bundles},
        date={2022},
     journal={arXiv preprint arXiv:2212.01360, to appear in Kyoto J.~Math.},
}

\bib{Hormbook}{book}{
      author={H\"{o}rmander, Lars},
       title={An introduction to complex analysis in several variables},
     edition={Third},
      series={North-Holland Mathematical Library},
   publisher={North-Holland Publishing Co., Amsterdam},
        date={1990},
      volume={7},
        ISBN={0-444-88446-7},
      review={\MR{1045639}},
}

\bib{KoikeUehara}{article}{
      author={Koike, Takayuki},
      author={Uehara, Takato},
       title={A gluing construction of {K3} surfaces},
        date={2019},
     journal={arXiv preprint arXiv:1903.01444},
         url={https://arxiv.org/abs/1903.01444},
}

\bib{Lazarsfeld1}{book}{
      author={Lazarsfeld, Robert},
       title={Positivity in algebraic geometry. {I}},
      series={Ergebnisse der Mathematik und ihrer Grenzgebiete. 3. Folge. A
  Series of Modern Surveys in Mathematics [Results in Mathematics and Related
  Areas. 3rd Series. A Series of Modern Surveys in Mathematics]},
   publisher={Springer-Verlag, Berlin},
        date={2004},
      volume={48},
        ISBN={3-540-22533-1},
         url={https://doi.org/10.1007/978-3-642-18808-4},
        note={Classical setting: line bundles and linear series},
      review={\MR{2095471}},
}

\bib{Mat}{article}{
      author={Matsumura, S.},
       title={An injectivity theorem with multiplier ideal sheaves of singular
  metrics with transcendental singularities},
        date={2018},
        ISSN={1056-3911,1534-7486},
     journal={J. Algebraic Geom.},
      volume={27},
      number={2},
       pages={305\ndash 337},
         url={https://doi.org/10.1090/jag/687},
      review={\MR{3764278}},
}

\bib{Ueda}{article}{
      author={Ueda, Tetsuo},
       title={On the neighborhood of a compact complex curve with topologically
  trivial normal bundle},
        date={1982/83},
        ISSN={0023-608X},
     journal={J. Math. Kyoto Univ.},
      volume={22},
      number={4},
       pages={583\ndash 607},
         url={https://doi.org/10.1215/kjm/1250521670},
      review={\MR{685520}},
}

\end{biblist}
\end{bibdiv}


\end{document}